\documentclass[reqno]{amsart}
\usepackage{amsmath,amssymb,amsfonts,amsthm,longtable}
\usepackage{hyperref}

\newcommand{\GETOUT}[1]{}
\newcommand{\NBC}{\Delta_{\mathrm{BC}}}

\newcommand{\rank}{\mathrm{rank}}
\newcommand{\Mon}{\mathrm{Mon}}
\newcommand{\Hilb}{\mathrm{Hilb}}
\newcommand{\A}{\mathcal{A}}
\newcommand{\I}{\mathcal{I}}
\newcommand{\C}{\mathcal{C}}
\newcommand{\B}{\mathcal{B}}
\newcommand{\ci}{\mathrm{ci}}
\newcommand{\coc}{\mathrm{coc}}
\newcommand{\chara}{\mathrm{char}}
\newcommand{\dotcup}{\ensuremath{\mathaccent\cdot\cup}}

\newtheorem{theorem}{Theorem}
\newtheorem{lemma}[theorem]{Lemma}
\newtheorem{proposition}[theorem]{Proposition}
\newtheorem{corollary}[theorem]{Corollary}
\newtheorem{conjecture}[theorem]{Conjecture}
\newtheorem{question}[theorem]{Question}
\theoremstyle{definition}
\newtheorem{definition}{Definition}
\newtheorem{remark}[theorem]{Remark}
\newtheorem{example}{Example}

\title{A Generalization of NBC Bases to Broken Circuit Complexes of Matroids}
\author[A. Egilsson]{Andri Egilsson}
\author[M. Kubitzke]{Martina Kubitzke}

\address{A. Egilsson and M. Kubitzke: The Mathematics Institute,
Reykjav\'ik University, 103 Reykjav\'ik, Iceland }
\thanks{Both authors were supported by grant no. 100038011 from the Icelandic Research Fund.}

\begin{document}

\maketitle

\begin{abstract}
Brown has shown that the Stanley-Reisner ring of the broken circuit complex of a graph has a linear system of parameters which is defined in terms of the circuits and cocircuits of the graph. Later on Brown and Sagan conjectured a special set of monomials -- a so-called NBC basis -- described in terms of the circuits and cocircuits of the graph to be a monomial basis for the corresponding quotient of the Stanley-Reisner ring and proved this to be true for theta and phi graphs. We generalize the aforementioned linear system of parameters to broken circuit complexes of regular matroids and transfer the notion of NBC bases to the general setting of regular matroids. We are able to obtain the analogous results to the ones of Brown and Sagan in this more general context. We show a deletion-contraction axiom for the existence of NBC bases. Using this results we identify two infinite classes of matroids which have NBC bases and which are the matroid theoretic analogue of theta and phi graphs.
\end{abstract}

\section{Introduction}
The main interest of this paper lies in the study of linear systems of parameters (l.s.o.p. for short) for the Stanley-Reisner ring of the broken circuit complex of a regular matroid as well as on the study of monomial bases for the quotient of the Stanley-Reisner ring with this particular l.s.o.p.
The motivation originates from work of Brown, Colbourn, Sagan and Wagner \cite{Brown,BrownColbournWagner,BrownSagan}. In \cite{Brown} and \cite{BrownColbournWagner} Brown and Brown, Colbourn and Wagner, respectively, construct a special l.s.o.p. for the Stanley-Reisner ring of 
the broken circuit complex (NBC complex for short) and the cographic matroid of a graph, respectively. This l.s.o.p. has a description in terms of the circuits and cocircuits of the underlying graph.
Built up on those results Brown and Sagan define in \cite{BrownSagan} a distinguished set of monomials in the Stanley-Reisner ring of the NBC complex of a graph which is given in terms of the circuits and cocircuits of the underlying graph. They say that a graph has an \emph{NBC basis} if this particular set is a monomial basis for the quotient of the Stanley-Reisner ring of the broken circuit complex with the 
l.s.o.p. from \cite{Brown}. They prove that under certain conditions on the graph this one has 
an NBC basis if its deletion and contraction in an edge have this property. Using this result they identify two classes of graphs -- so-called \emph{phi} and \emph{theta graphs} -- admitting NBC bases.
Our work is inspired by their suggestion to try to generalize their constructions to representable matroids.

This paper is structured as follows.
In Section \ref{sec:background} we recall some theory about simplicial complexes and matroids. Here, we focus on the broken circuit complex $\NBC(M)$ of a matroid $M$ as well as on the class of regular matroids.\\
In Section \ref{sec:lsop} it is shown that the construction of the l.s.o.p. for the NBC complex of a graph also works in the more general setting of NBC complexes of regular matroids. In order to get this result we first need to introduce the (fundamental) circuit and cocircuit incidence matrices of a matroid and transfer some results for the corresponding matrices for graphs to our setting. In Proposition \ref{prop:rank} the rank of those matrices is computed and Theorem \ref{theo:nonsing} provides a criterion for certain submatrices of a particularly signed cocircuit matrix of the matroid to be non-singular.
Those results together with a characterization of l.s.o.p.s for matroids, stated in Corollary \ref{cor:lsop}, enable us to prove that the NBC, as well as the independence complex, of a regular matroid have a special l.s.o.p. which is given in terms of its circuits and cocircuits (see Theorem \ref{th:lsop}), thereby generalizing the aforementioned result of \cite{Brown}.\\
In Section \ref{Sect:NBCbases} we try to find monomial bases for the quotient of the Stanley-Reisner ring of the NBC complex with the l.s.o.p. defined in Section \ref{sec:lsop}. Following \cite{BrownSagan} we introduce the definition of an NBC basis, see Definition \ref{def:nbc} within this section. Theorem \ref{Th:DelContr} which is the main result of this section provides a deletion-contraction axiom for the existence of NBC bases. This result is the matroid-theoretic analogue of Theorem 3.2 in \cite{BrownSagan}.\\
In the last part of this paper, i.e., in Section \ref{sec:ThetaPhi}, we are seeking for classes of matroids which admit NBC bases. We are able to identify two such classes,
both being parallel connections of certain uniform matroids.
The first class are iterated parallel connections of uniform matroids of rank $n_i-1$ on an $n_i$-element set with respect to a single basepoint, see Theorem \ref{Th:phi}. Those matroids are the analogue of the phi graphs considered in \cite[Theorem 4.2]{BrownSagan}.
In Theorem \ref{Th:theta} we prove the existence of NBC bases for iterated parallel connections with respect to different basepoints, thereby generalizing Theorem 4.4 in \cite{BrownSagan} for theta graphs.
We conclude this section with some examples and a brief discussion about which classes of matroids might possess NBC bases.

\section{Background}\label{sec:background}
\subsection{Simplicial complexes} \label{subsect:SimplicialComplexes}
We use $[n]$ to denote the set $\{1,\ldots,n\}$ and $k$ to denote an arbitrary field.
A \emph{simplicial complex} $\Delta$ on vertex set $[n]$ is a collection of subsets of $2^{[n]}$ such that $\emptyset\in \Delta$ and $F\in \Delta$ and
$G\subseteq F$ implies that $G\in \Delta$. The elements of $\Delta$ are called \emph{faces} and the \emph{dimension} of a face is its cardinality minus $1$. The \emph{dimension} of $\Delta$ is the maximal dimension of its faces. Throughout this paper we use $d-1$ to denote the dimension of a simplicial complex. We call a simplicial complex \emph{pure} if its maximal faces -- also referred to as \emph{facets} -- are of the same dimension. It is common to associate to a simplicial complex its so-called \emph{$f$-vector} $f^{\Delta}=(f_{-1}^{\Delta},f_0^{\Delta},\ldots,f_{d-1}^{\Delta})$, where
\begin{equation*}
f_i^{\Delta}=|\{F\in\Delta~|~\dim F=i\}|.
\end{equation*}
Instead of considering the $f$-vector of a simplicial complex it is sometimes more convenient to look at its \emph{$h$-vector} $h^{\Delta}=(h_0^{\Delta},h_1^{\Delta},\ldots,h_{d}^{\Delta})$ which is defined by:
\begin{equation*}
\sum_{i=0}^{d}h_i^{\Delta}t^i=\sum_{i=0}^{d}f_{i-1}^{\Delta}t^i(1-t)^{d-i}.
\end{equation*}
The polynomial $h^{\Delta}(t)=\sum_{i=0}^{d}h^{\Delta}_i t^{d-i}$ is referred to as the \emph{$h$-polynomial} of $\Delta$.
An algebraic object closely linked to a simplicial complex $\Delta$ which carries the same enumerative information is the \emph{Stanley-Reisner ring} $k[\Delta]$. More precisely, $k[\Delta]=k[x_1,\ldots,x_n]/I_{\Delta}$ where the \emph{Stanley-Reisner ideal} $I_{\Delta}$ is the ideal in $k[x_1,\ldots,x_n]$ generated by the (minimal) non-faces of $\Delta$, i.e.,
\begin{equation*}
I_{\Delta}=\left\langle \prod_{i\in F}x_i~|~F \notin \Delta\right\rangle.
\end{equation*}
It is widely used that the Stanley-Reisner ring and the $h$-vector of a simplicial complex are related in the following way, see e.g., \cite[Theorem 5.1.7]{BH-book}.
\begin{theorem} 
Let $\Delta$ be a $(d-1)$-dimensional simplicial complex and let
$h^{\Delta}=(h_{0}^{\Delta},h_1^{\Delta},\ldots,h_{d}^{\Delta})$ be its $h$-vector. Let further $\Hilb(k[\Delta],t)$ denote the Hilbert series of $k[\Delta]$. Then
\begin{equation} \label{Eq: Hilbert series}
\Hilb(k[\Delta],t)=\frac{h_0^{\Delta}+h_1^{\Delta}t+\cdots +h_d^{\Delta}t^d}{(1-t)^d}.
\end{equation}
\end{theorem}
In a considerable part of this paper we will explore and construct linear systems of parameters for special classes of simplicial complexes.
Recall that a \emph{linear system of parameters} for $k[\Delta]$ (\emph{l.s.o.p.} for short) is a set $\Theta=\{\theta_1,\ldots,\theta_d\}$ of linear forms in $k[\Delta]$ such that $k[\Delta]$ is a finitely generated $k[\theta_1,\ldots,\theta_{d}]-$module. (Note that we have $\dim k[\Delta]=d$.)\\
If $k$ is infinite such a system exists, see e.g., \cite[Lemma 5.2]{StanleyGreenBook} and assuming Cohen-Macaulayness of $\Delta$ over $k$ assures that 
\begin{equation*}
\Hilb(k[\Delta]/\Theta,t)=h_0^{\Delta}+h_1^{\Delta}t+\cdots +h_d^{\Delta}t^d,
\end{equation*}%
see e.g., \cite[Remark 4.1.11]{BH-book}.

\subsection{Matroid terminology}
In this section we review some theory about matroids, see e.g.,\cite{Oxley} for more details.\\
A \emph{matroid} $M$ is an ordered pair $(E(M),\I(M))$ consisting of a finite set $E(M)$ and a collection $\I(M)$ of subsets of $E(M)$ satisfying the following two conditions: 
\begin{itemize}
\item[(i)] $\I(M)$ is a simplicial complex on ground set $E(M)$, called \emph{independence complex}.
\item[(ii)] If $I_1$, $I_2\in \I(M)$ and $|I_1|<|I_2|$, then there exists $e\in (I_2\setminus I_1)$ such that $(I_1\cup\{e\})\in\I(M)$.
\end{itemize}
$E(M)$ is called the \emph{ground set} of $M$. Subsets of $E(M)$ are called \emph{independent} if they lie in $\I(M)$ and \emph{dependent} otherwise. If no confusion can occur we will often suppress $M$ from our notations and simply write $E$ and $\I$ for $E(M)$ and $\I(M)$, respectively.\\
A minimal dependent set in a matroid $M$ is called a \emph{circuit} of $M$ and we use $\C(M)$ to denote the set of circuits of $M$. In particular, a circuit of cardinality $1$ is called a \emph{loop}.
Similarly, the maximal independent sets of $M$ are called the \emph{bases} of $M$ and the set of bases of $M$ is denoted by $\B(M)$. Note that either of the sets $\I(M)$, $\C(M)$ and $\B(M)$ determines the matroid.\\
A widely studied subcomplex of the independence complex of a matroid $M$ is the so-called \emph{broken circuit complex} of $M$, see e.g., \cite{SwartzMatroid}, Chapter $7$ by Bj\"orner in \cite{White} and \cite{Brylawski}.
It is constructed in the following way:

First endow the ground set of $M$ with a linear ordering, i.e., assume that $E(M)=\{e_1<\cdots <e_n\}$. From each circuit $C\in \C(M)$ remove its minimal element. This yields the so-called \emph{broken circuit} $\overline{C}=C\setminus\{\min C\}$. Note that different circuits may produce the same broken circuit. Finally, the \emph{broken circuit complex of $M$} (\emph{NBC complex} for short) is
\begin{equation*}
\NBC(M)=\{F\subseteq E(M)~|~ F\mbox{ does not contain a broken circuit of } M\}.
\end{equation*}
Even though this definition depends on the chosen ordering of $E(M)$ the face numbers of this simplicial complex do not, see Section 7 in \cite{White}.
It is easy to check that $\NBC(M)$ is indeed a pure simplicial complex whose facets are the bases of $M$ which do not contain a broken circuit. In particular $\NBC(M)$ is a subcomplex of $\I(M)$ of the same dimension. Easier than to describe the faces of $\NBC(M)$ is to find an expression for its Stanley-Reisner ideal:
\begin{equation*}
I_{\NBC(M)}=\left\langle \prod_{i\in\overline{C}}x_i~\Big|~C\in\C(M)\right\rangle.
\end{equation*}
As for many combinatorial objects there exists a notion of duality one can define the dual of a matroid, likewise.
More precisely, given a matroid $M=(E(M),\I(M))$ we set
\begin{equation*}
\B^*(M)=\{E(M)\setminus B~|~B\in \B(M)\}
\end{equation*}
It can be shown that $\B^*(M)$ is the set of bases of a matroid $M^*$ on ground set $E(M)$, the \emph{dual matroid} of $M$. Circuits, loops,  independent sets and bases of $M^*$ are called \emph{cocircuits}, \emph{coloops}, \emph{coindependent sets} and \emph{cobases} of $M$, respectively. For our considerations, the circuits and cocircuits of $M$ are the crucial objects. In the following we therefore evolve a bit deeper into the subject of circuits, cocircuits and relations between those two classes of objects.\\
Let us first fix a particular basis $B\in \B(M)$. If $e\in E(M)\setminus B$ it can be verified that the set $B\cup\{e\}$ contains a unique circuit which contains $e$, see e.g., \cite[Proposition 1.1.6]{Oxley}. We call this circuit the \emph{fundamental circuit of $e$ with respect to $B$} and denote it by $\ci(B,e)$.
Similarly, if $b\in B$ then there exists a unique cocircuit of $M$ contained in $(E(M)\setminus B)\cup\{b\}$ which contains $b$. This cocircuit is called
the \emph{fundamental cocircuit of $b$ with respect to $B$} and is denoted by $\coc(B,b)$. The fundamental circuits and cocircuits of $M$ with respect to a basis $B$ are related in the following way.
\begin{lemma} \cite[Lemma 7.3.1]{White} \label{lem:CoCircuits}
Let $M=(E(M),\I(M))$ be a matroid and let $B$ be a basis of $M$. Let further $e\notin B$ and $b\in B$ be two elements of the matroid. Then
\begin{equation*}
b\in \ci(B,e)\qquad\Leftrightarrow \qquad e\in \coc(B,b).
\end{equation*}
\end{lemma}
We will intensively make use of this connection between fundamental circuits and cocircuits. The following proposition provides
another characterization of cocircuits which is sometimes easier to apply.
\begin{proposition}\cite[Proposition 2.1.16]{Oxley}\label{prop:cocircuits}
Let $M=(E(M),\I(M))$ be a matroid. Then the cocircuits of $M$ are the minimal subsets of $E(M)$ intersecting any basis of $M$ non-trivially.
\end{proposition}
Our work focuses on representable matroids.
Assume that $A$ is an $(m\times n)$-matrix with entries in $k$ and let $E$ be its set of column labels. Consider all subsets of $E$ such that the corresponding columns of $A$ are linearly independent over $k$ and let $\I$ be the collection of those sets. It is easy to see that $M[A]=(E,\I)$ is a matroid, see \cite[Proposition 1.1.1]{Oxley}. In general,
a matroid $M$ is said to be \emph{representable} over $k$ if and only if $M$ is isomorphic to some $M[A]$ for some matrix $A$ over $k$. 
I.e., there exist a matrix $A=(a_{ij})\in k^{m\times n}$ and a map
\begin{align*}
\Phi:\; E(M)&\rightarrow k^m\\
e_i &\mapsto \left(
\begin{array}{ccc}
a_{1i}\\
\vdots\\
a_{mi}
\end{array}
\right)
\end{align*}
such that $I\in \I(M)$ if and only if $\Phi(I)$ is a linearly independent set of columns over $k$.
The map $\Phi$ is called a \emph{representation} of $M$ over $k$.
If a matroid $M$ can be represented by a matrix whose square submatrices all have determinant $-1$, $0$ or $1$ then it is called a \emph{regular}. Several other characterizations of regular matroids are known.
\begin{proposition}\cite[Theorem 6.6.3]{Oxley}
Let $M$ be a matroid. Then the following are equivalent:
\begin{itemize}
\item[(i)] $M$ is regular.
\item[(ii)] $M$ is representable over every field.
\item[(iii)] $M$ is representable over $\mathbb{F}_2$, the $2$-element field, and over another field $k$ with $\chara(k)\neq 2$.
\end{itemize}
\end{proposition}
More background on matroids can be found in \cite{Oxley}.

\section{Special linear systems of parameters for regular matroids} \label{sec:lsop}
The aim of this section is to construct a l.s.o.p. for the Stanley-Reisner ring of the independence as well as the NBC complex of a regular matroid which can be described in terms of the circuits and cocircuits of the matroid.
The given construction is a generalization of the l.s.o.p. given in \cite{BrownColbournWagner} and \cite{Brown} for cographic matroids and graphic matroids as well as for NBC complexes of graphs, respectively. The ideas for the proof are inspired by their ideas and the main task is to translate results from graph theory into the language of matroids and to prove their analogues in this context.

We first introduce the notion of $2$-connected matroids and briefly explain why we can restrict ourselves to the consideration of this special class when searching l.s.o.p.s.

\subsection{\texorpdfstring{Restriction to $2$-connected matroids}{Restriction to 2-connected matroids}}
$2$-connectedness for matroids can be considered the matroid-theoretical analogue of connectedness for graphs.
For a matroid $M=(E(M),\I(M))$ and $e\in E(M)$ we set 
\begin{equation*}
 C(e):=\{e\}\cup\{f\in E(M)~|~M \mbox{ has a circuit containing both, }e\mbox{ and }f\}.
\end{equation*}
Then one can write $E(M)=\bigcup_{i=1}^r C(e_i)$ for elements $e_1,\ldots,e_r\in E(M)$. If there exists a single $e\in E(M)$ such that $C(e)=E(M)$ we call $M$ \emph{$2$-connected} and in this situation each element of the ground set has the required property.
Otherwise, $M$ is \emph{disconnected} and $C(e_1),\ldots,C(e_r)$ are the \emph{connected components of $M$}. It is straightforward to show that the latter ones are pairwise disjoint which implies that the NBC complex of $M$ equals the join of the NBC of its connected components $M_1,\ldots, M_r$, i.e.,
\begin{equation} \label{Eq:NBC}
\NBC(M)=\ast_{i=1}^r \NBC(M_i).
\end{equation}
Thus for the Stanley-Reisner rings it holds that
\begin{equation*}
k[\NBC(M)]=k[\NBC(M_1)]\otimes_k \cdots \otimes_k k[\NBC(M_r)].
\end{equation*}
If $\Theta_i$ is a l.s.o.p. for $k[\NBC(M_i)]$ for $1\leq i\leq r$, then
\begin{equation*}
k[\NBC(M)]/\left( \Theta_1\dotcup\cdots\dotcup\Theta_r\right)\cong k[\NBC(M_1)]/\Theta_1\otimes_k\cdots \otimes_k k[\NBC(M_r)]/\Theta_r
\end{equation*}
which shows that $\Theta_1\dotcup\cdots \dotcup\Theta_r$ is an l.s.o.p. for $\NBC(M)$.
When constructing an l.s.o.p. for the NBC complex of $M$ it is therefore enough to consider its $2$-connected components.

\subsection{The circuit and cocircuit incidence matrix of a matroid} \label{subsect:CircuitCocircuit}

It is common to associate to a given graph its circuit and cocircuit incidence matrices -- their entries essentially indicating if a certain edge lies in a particular circuit and cocircuit, respectively. To proceed one step further, the graph is usually assigned an overall orientation and accessorily each of its circuits and cocircuits is directed separately. Given those orientations one signs the entries of the incidence matrices with $+1$ or $-1$ depending on whether the overall orientation of an edge and its orientation in a particular circuit and cocircuit coincide or whether they do not. Assuming the same column labelling the rows of the two incidence matrices can be shown to be orthogonal to each other, see e.g., \cite[Theorem 6.6]{Thulasiraman}. We will refer to this property as the \emph{orthogonal property}.
Mimicking the described construction for matroids yields the following definitions.
Consider a rank $r$ matroid $M$ on ground set $E(M)=\{e_1,\ldots,e_n\}$ and let $B$ be a basis of $M$.
Label the fundamental circuits of $M$ w.r.t. $B$ with $C_1,\ldots,C_{n-r}$. 
The remaining circuits of $M$ are labelled with $C_{n-r+1},\ldots, C_m$.
Let now $\A_M(\C)$ be an $(m\times n)$-matrix whose rows and columns are indexed with the circuits $C_1,\ldots,C_m$ and the elements $e_1,\ldots ,e_n$ of $M$, respectively. The entry $a_{ij}^M$ is defined by
\begin{equation*}
a_{ij}^M=
\begin{cases}
1,\qquad \mbox{if } e_j\in C_i\\
0,\qquad \mbox{otherwise}.
\end{cases}
\end{equation*}
The matrix $\A_M(\C)$ is called the \emph{circuit incidence matrix of $M$ w.r.t. $B$}.

In order to define the cocircuit incidence matrix of $M$ we label the fundamental cocircuits of $M$ w.r.t. $B$ with $C^*_1,\ldots , C^*_r$ and the remaining cocircuits with $C^*_{r+1},\ldots,C^*_s$. Let $\A_M(\C^*)$ be an $(s\times n)$-matrix whose rows and columns are labelled with the cocircuits $C^*_1,\ldots, C^*_s$ and the elements $e_1,\ldots ,e_n$, respectively. The entry $a_{ij}^*$ is defined as
\begin{equation*}
a_{ij}^*=
\begin{cases}
1,\qquad \mbox{if } e_j\in C^*_i\\
0,\qquad \mbox{otherwise}.
\end{cases}
\end{equation*}
The matrix $\A_M(\C^*)$ is called the \emph{cocircuit incidence matrix of $M$ w.r.t. $B$}. \\
Note that cocircuit and circuit incidence matrices w.r.t. different bases and orderings of $E(M)$ differ only in the order of their rows and columns, respectively.
For simplicity we will therefore just speak of the circuit and cocircuit incidence matrices of $M$ without mentioning the particular basis used.\\
Imposing a signing on the entries of those two matrices, in general we cannot expect that -- as for graphs -- the orthogonality property holds. However, if there exist signings $\widetilde{\A}_M(C)$ and $\widetilde{\A}_M(C^*)$ of $\A_M(\C)$ and $\A_M(\C^*)$, respectively such that over $\mathbb{R}$ the orthogonality property holds, i.e.,
$\widetilde{\A}_M(C)\cdot (\widetilde{\A}_M(C^*))^T=0$ then $M$ is called \emph{signable}. Here as well as in the sequel, the tilde indicates that the matrices are signed. Signings of the circuit and cocircuit matrices such that the orthogonality property holds over $k$ will be called \emph{compatible signings over $k$}. 
Since those signings always come in pairs we will use the term \emph{admissible signing} if we only consider a signing of one of $\A_M(\C)$ and $\A_M(\C^*)$ but if there exists a signing of the other matrix such that the orthogonality property holds.
The following proposition brings us back to the class of matroids considered so far, i.e., regular matroids.
\begin{proposition} \cite[Proposition 13.4.5]{Oxley}
A matroid is signable if and only if it is regular.
\end{proposition}
\begin{remark}
Let $M$ be a regular matroid and let $\widetilde{\A}_M(\C)$ and $\widetilde{\A}_M(\C^*)$ be compatible signed circuit and cocircuit matrices of $M$, respectively. If $p_{C}=(p_1,\ldots,p_n)$ and $q_{C^*}=(q_1,\ldots,q_n)$ are rows of $\widetilde{\A}_M(\C)$ and $\widetilde{\A}_M(\C^*)$, respectively, it holds that
\begin{equation}\label{eq:comp}
p_C\cdot (q_{C^*})^T=\sum_{i=1}^np_iq_i=0
\end{equation}
Since $p_{C}$, $q_{C^*}\in\{-1,0,1\}$ there
must be as many positive terms in Equation (\ref{eq:comp}) as negative terms. In particular, the number of non-zero terms in the above sum is even. Over a field of characteristic $2$ each signed circuit and cocircuit matrix coincides with the corresponding unsigned matrix. By the previous discussion, those matrices are thus compatible in this case.
\end{remark}
Rather than being interested in the complete circuit and cocircuit matrix of $M$ we mostly consider submatrices which consist of the rows corresponding to the fundamental circuits and cocircuits of $M$ w.r.t. a basis $B$. Those matrices are referred to as the \emph{fundamental circuit} and \emph{cocircuit incidence matrix w.r.t. $B$} and denoted by $\A_M^{B}(\C)$ and $\A_M^{B}(\C^*)$, respectively. 
Using compatible signings we can determine the rank of those matrices.
\begin{proposition}\label{prop:rank}
Let $M$ be a regular matroid on ground set $E(M)=\{e_1,\ldots,e_n\}$ of rank $r$. Let $B\in\B(M)$ be a basis of $M$ and let $\widetilde{\A}_M(\C)$ and $\widetilde{\A}_M(\C^*)$ be compatible signed circuit and cocircuit incidence matrices of $M$ with respect to $B$, respectively. Then
\begin{equation} \label{eq:rank1}
\rank\; \widetilde{\A}_M^B(\C)=\rank\; \widetilde{\A}_M(\C)=n-r
\end{equation}
and
\begin{equation} \label{eq:rank2}
\rank\; \widetilde{\A}_M^B(\C^*)=\rank\; \widetilde{\A}_M(\C^*)=r.
\end{equation}
\end{proposition}
\begin{proof}
Throughout the proof we assume that $B=\{e_{n-r+1},\ldots,e_{n-1},e_n\}$ and that the columns of each of the four considered matrices are ordered according to $e_1,\ldots ,e_n$. Furthermore, we order the rows of $\widetilde{\A}_M(\C)$ such as $\widetilde{\A}_M^B(\C)$ and $\widetilde{\A}_M(\C^*)$ such as $\widetilde{\A}_M^B(\C^*)$ such that the $i^{th}$ row corresponds to the fundamental circuit $\ci(B,e_{i})$ and the fundamental cocircuit $\coc(B,e_{n-r+i})$ for $1\leq i\leq n-r$ and $1\leq i\leq r$, respectively.

We first compute the rank of $\widetilde{\A}_M^B(\C)$ and $\widetilde{\A}_M^B(\C^*)$. In the chosen orderings the first $n-r$ columns of $\widetilde{\A}_M^B(\C)$ are a diagonal matrix having entries $+1$ and $-1$, thus this submatrix is non-singular, showing that $\rank\;\widetilde{\A}_M^B(\C)=n-r$. Similarly, the last $r$ columns of $\widetilde{\A}_M^B(\C^*)$ are a diagonal matrix having entries $+1$ and $-1$ in the diagonal. Thus, this submatrix is non-singular as well and therefore $\widetilde{\A}_M^B(\C^*)$ is of rank $r$.

Next we show that $\rank\;\widetilde{\A}_M(\C^*)=r$. Note that being $\widetilde{\A}_M^B(\C^*)$ a submatrix of $\widetilde{\A}_M(\C^*)$ implies
\begin{equation}\label{eq:rank}
\rank\; \widetilde{\A}_M(\C^*)\geq \rank \;\widetilde{\A}_M^B(\C^*)=r.
\end{equation}
By assumption $\widetilde{\A}_M^B(\C)$ and $\widetilde{\A}_M^B(\C^*)$ have the orthogonal property. Since multiplying rows with a non-zero scalar preserves this property, we may assume -- using the above discussion -- that
\begin{equation*}
\widetilde{\A}_M^B(\C)=[E_{n-r},\A_{2}(\C)]\qquad \mbox{and} \qquad \widetilde{\A}_M^B(\C^*)=[\A_{1}(\C^*),E_r],
\end{equation*}
where $E_j$ denotes the $(j\times j)$-unit matrix.
The orthogonality property now implies
\begin{equation*}
0=[E_{n-r},\A_{2}(\C)]\cdot \left[
\begin{array}{cc}
\A_{1}(\C^*)^T\\
E_r
\end{array}
\right]
=E_{n-r}\cdot \A_{1}(\C^*)^T+\A_{2}(\C)\cdot E_r,
\end{equation*}
i.e., 
\begin{equation}\label{eq:orth}
\A_{2}(\C)=-\A_{1}(\C^*)^T.
\end{equation}
Let $\alpha=(\alpha_1,\ldots,\alpha_{n-r},\alpha_{n-r+1},\ldots,\alpha_n)$ be any row vector of $\widetilde{\A}_M(\C^*)$. The claim follows if we show that $\alpha$ can be written as a linear combination of the rows of $\widetilde{\A}_M^B(\C^*)$. By the orthogonality property it holds that
\begin{equation*}
0=\alpha\cdot (\widetilde{\A}_M^B(\C))^T=(\alpha_1,\ldots ,\alpha_{n-r},\alpha_{n-r+1},\ldots,\alpha_{n})\cdot\left[
\begin{array}{cc}
E_{n-r}\\
\A_{2}(\C)^T
\end{array}
\right].
\end{equation*}
Using (\ref{eq:orth}) we deduce that 
\begin{align*}
(\alpha_{1},\ldots,\alpha_{n-r})&=-(\alpha_{n-r+1},\ldots,\alpha_n)\cdot \A_{2}(\C)^T\\
&=(\alpha_{n-r+1},\ldots,\alpha_n)\cdot \A_{1}(\C^*).
\end{align*}
We now conclude that 
\begin{equation*}
\alpha=(\alpha_{n-r+1},\ldots,\alpha_n)\cdot [\A_{1}(\C^*),E_r]=(\alpha_{n-r+1},\ldots,\alpha_n)\cdot\widetilde{\A}_M^B(\C^*).
\end{equation*}
Thus, any row vector of $\widetilde{\A}_M(\C^*)$ is given as a linear combination of the rows of $\widetilde{\A}_M^B(\C^*)$ and hence
\begin{equation*}
\rank\; \widetilde{\A}_M(\C^*)\leq \rank\; \widetilde{\A}_M^B(\C^*)=r.
\end{equation*}
This finishes the proof of Equation (\ref{eq:rank2}). The second part of Equation (\ref{eq:rank1}) follows similarly.
\end{proof}
In the above proof the computation of the rank of the fundamental circuit and cocircuit matrices did not use the orthogonality property. Thus, this part remains true for all signings of those matrices.

The next theorem which is very much in the flavor of Theorem 6.10. in \cite{Thulasiraman} provides a criterion for certain submatrices of an admissible signed cocircuit matrix of $M$ to be non-singular.
\begin{theorem}\label{theo:nonsing}
Let $M$ be a regular matroid on ground set $E(M)=\{e_1,\ldots,e_n\}$ of rank $r$. Let $\widetilde{\A}_M(\C^*)$ be an admissible signed cocircuit incidence matrix of $M$ and let $\A$ be an $(r\times n)$-submatrix of $\widetilde{\A}_M(\C^*)$ having full rank $r$. Then:\\
A square submatrix of $\A$ of size $(r\times r)$ is non-singular if and only if the elements of $M$ corresponding to the columns of the submatrix of $\A$ form a basis of $M$.
\end{theorem}
\begin{proof}
We first show the necessity part. Assume that $B=\{e_{n-r+1},\ldots,e_n\}\in \B(M)$ is a basis of $M$ and let the columns of $\A$ be labelled such that the $i^{th}$ column corresponds to $e_i$. If we write
\begin{equation*}
\A=[\A_{1}, \A_{2}]
\end{equation*}
-- where the labels of the columns of $\A_2$ correspond to $B$ -- we need to show that $\A_2$ is non-singular. Consider the fundamental cocircuit matrix $\A_M^{B}(\C^*)$ w.r.t. $B$.
Since $\A$ is a submatrix of $\widetilde{\A}_M(\C^*)$ it follows from Proposition \ref{prop:rank} that 
there exists an $(r\times r)$-matrix $D$ such that 
\begin{equation} \label{eq:submatrix}
\A=[\A_1,\A_2]=D\cdot\widetilde{\A}_M^{B}(\C^*).
\end{equation}
The same arguments as in the proof of Proposition \ref{prop:rank} show that we may assume that
\begin{equation*}
\widetilde{\A}_M^{B}(\C^*)=[\widetilde{\A}_1(\C^*),E_r].
\end{equation*}
Equation (\ref{eq:submatrix}) now implies that
\begin{equation*}
\A=[\A_1,\A_2]=D\cdot[\widetilde{\A}_1(\C^*),E_r]=[D\cdot\widetilde{\A}_1(\C^*),D]
\end{equation*}
and in particular $\A_2=D$.
Since both, $\A$ and $\A_M^{B}(\C^*)$, have maximal rank, $D$, i.e, $\A_2$ has to be non-singular as well. This shows the claim.

We now show the sufficiency part. Let $\A_1$ be a non-singular $(r\times r)$-submatrix of $\A$. Let the columns of $\A$ be rearranged such that $\A=[\A_1,\A_2]$. Without loss of generality let $\widetilde{E}=\{e_1,\ldots,e_r\}$ be the elements of $M$ corresponding to the columns of $\A_1$. We show that $\widetilde{E}$ does not contain a circuit of $M$. Assume for the contrary that there exists a circuit $C\in \C(M)$ such that $C\subseteq\widetilde{E}$. If $\beta=(\beta_1,\ldots,\beta_n)$ is the row -- labelled  $C$ -- of the signed circuit incidence matrix of $M$, compatible with $\widetilde{\A}_M(\C^*)$ it must hold that $\beta_{r+1}=\cdots =\beta_n=0$. The orthogonality property implies
\begin{align*}
0&=\beta\cdot[\A_1,\A_2]^T\\
&=(\beta_1,\ldots,\beta_r)\cdot \A_1^T+(\beta_{r+1},\ldots,\beta_n)\cdot \A_2^T\\
&=(\beta_1,\ldots,\beta_r)\cdot \A_1^T.
\end{align*}
This yields a linear dependence relation between the columns of $\A_1$ which contradicts the assumption. Thus, $\widetilde{E}$ has to be an independent set of $M$ and $\rank(M)=r=|\widetilde{E}|$ implies that it is a basis.
\end{proof}

\subsection{Linear systems of parameters for regular matroids}
In Section \ref{subsect:SimplicialComplexes} we have introduced the notion of a l.s.o.p. for the Stanley-Reisner of a simplicial complex. For practical purposes it is often easier to use the following characterization of those systems, due to Stanley \cite{StanleyCM}.
\begin{proposition}\label{prop:lsop}
Suppose $\Delta$ is a $(d-1)$-dimensional simplicial complex and
\begin{equation*}
\theta_1=\sum_{i=1}^n a_{1i}x_i,\qquad\ldots,\qquad\theta_d=\sum_{i=1}^n a_{di}x_i
\end{equation*}
are homogeneous elements of degree $1$ in $k[\Delta]$. Let $A=(a_{ij})\in k^{d\times n}$ and for a facet $F\in \Delta$ let $A_F$ be the $(d\times d)$-submatrix of $A$ whose columns correspond to the vertices of $F$. Then $\{\theta_1,\ldots,\theta_d\}$ is an l.s.o.p. for $k[\Delta]$ if and only if for every facet $F$ of $\Delta$ the matrix $A_F$ has rank $d$.
\end{proposition}
For the independence complex of a matroid, Brown, Colbourn and Wagner \cite{BrownColbournWagner}
reformulated this criterion the following way.
\begin{corollary}\label{cor:lsop}
Let $M=(E(M),\I(M))$ be a matroid of rank $r$ and $E(M)=\{e_1,\ldots,e_n\}$. Let 
\begin{equation*}
\theta_1=\sum_{i=1}^n a_{1i}x_i,\qquad\ldots,\qquad\theta_r=\sum_{i=1}^n a_{ri}x_i
\end{equation*}
be homogeneous elements of degree $1$ in $k[\I(M)]$. If the map
\begin{equation*}
\Phi:\;E(M)\;\rightarrow \; k^r:\;e_i\;\mapsto (a_{1i},\ldots,a_{ri})^T
\end{equation*}
is a representation of $M$ over $k$ then $\{\theta_1,\ldots,\theta_r\}$ is an l.s.o.p. for $k[\I(M)]$.
\end{corollary}
Brown, Colbourn and Wagner \cite{BrownColbournWagner} as well as Brown \cite{Brown} used this criterion to construct a l.s.o.p. for cographic matroids and the independence as well as the NBC complex of a graph, respectively. Our construction of a l.s.o.p. for the independence as well as NBC complex of a matroid closely follows those lines and uses similar ideas.

In the sequel, we consider the following setup. Let $M=(E,\I)$ be a regular matroid of rank $r$.
Let $B\in\B(M)$ be a basis and $\A_M^B(\C^*)=(a_{ij}^*)$ be an admissible signed fundamental cocircuit matrix w.r.t. $B$.
Assume that $E=\{e_1,\ldots,e_n\}$ and $B=\{e_{n-r+1},\ldots,e_n\}$. Let columns and rows of $\A_M^B(\C^*)$ be ordered accordingly.
To each fundamental cocircuit $\coc(B,e_j)$ we associate a linear form
\begin{equation}\label{eq:lsop}
\theta_{e_j}=\sum_{e_i\in \coc(B,e_j)}a^*_{ji} x_i
\end{equation}
in $k[x_1,\ldots,x_n]$.
The result analogous to Theorems 3 in \cite{BrownColbournWagner} and \cite{Brown} is the following.
\begin{theorem} \label{th:lsop}
Let $M=(E,I)$ be a regular matroid of rank $r$ and let $B\in \B(M)$ be a basis.
Then the set $\Theta^B=\{\theta_e~|~e\in B\}$ ($\theta_e$ as defined in (\ref{eq:lsop})) is a l.s.o.p. for $k[\I(M)]$ and $k[\NBC(M)]$.
\end{theorem}
\begin{proof}
Throughout the proof we may assume that the matroid $M$ is $2$-connected.
In Proposition \ref{prop:rank} we had shown that $\rank\; \A_M^B(C^*)=r$. Theorem \ref{theo:nonsing} states that an $(r\times r)$-submatrix of $ \A_M^B(C^*)$ is non-singular if and only if the elements of $M$ corresponding to the columns of this matrix are a basis of $M$. But this is nothing else than to require that $\A_M^B(C^*)$ is a representation of the matroid $M$. 
Thus, by Corollary \ref{cor:lsop} it follows that $\Theta^B$ is a l.s.o.p. for $k[\I(M)]$.\\
For the second part of the claim, note that $\dim\NBC(M)=\dim \I(M)=r-1$ and observe that every facet of $\NBC(M)$ is a facet of $\I(M)$. Applying Proposition \ref{prop:lsop} first to $\I(M)$ and then to $\NBC(M)$ we conclude that $\Theta^B$ is a l.s.o.p. for $k[\NBC(M)]$.
\end{proof}

\section{Monomial bases for matroids} \label{Sect:NBCbases}
Throughout this section let $M$ be a regular matroid of rank $r$ and let $B\in\B(M)$ be a basis. Further, let $\Theta^B$ be the l.s.o.p. of $k[\NBC(M)]$ defined in (\ref{eq:lsop}) and set $k(\NBC(M))=k[\NBC(M)]/\Theta^B$.\\
Our aim is to find a basis of this quotient ring which has a description in terms of the combinatorics of $M$.
After having defined a candidate for such a basis we introduce \emph{NBC bases} and show a deletion-contraction axiom for their existence.
Our work is inspired by the article \cite{BrownSagan} of Brown and Sagan where NBC bases for NBC complexes of graphs are defined. There it is indicated that an analogous construction could be meaningful for regular matroids. In what follows we carry out exactly this idea.

\subsection{The notion of NBC bases}
We first would like to remark that there exists a (non-monomial) basis of $k(\NBC(M))$ which does not make use of all variables. 
Let $E=\{e_1,\ldots,e_n\}$ of $M$ and assume that $B=\{e_{n-r+1},\ldots,e_n\}$.
For $e_j\in B$ consider the fundamental cocircuit $\coc(B,e_j)$ and the corresponding linear form $\theta_{e_j}\in \Theta^B$.
Since $\theta_{e_j}=0\in k(\NBC(M))$ it follows that
\begin{equation} \label{eq:red}
-a^*_{jj}x_j=\sum_{\substack{e_i\in \coc(B,e_j)\\i\neq j}}a^*_{ji} x_i.
\end{equation}
Since the right-hand side of (\ref{eq:red}) does only involve variables $x_i$ with $i\leq n-r$ we can substitute all variables $x_j\in k(\NBC(M))$ with $n-r+1\leq j \leq n$ by linear forms in $k[x_{1},\ldots,x_{n-r}]$.\\
For each circuit $C$ of $M$ let $p_{C}=p_C(x_1,\ldots,x_{n-r})$ be the polynomial obtained from $\prod_{e_j\in \overline{C}}x_j$ after performing those substitutions. (Recall that we use $\overline{C}$ to denote $C\setminus\{\min C\}$.) We set
\begin{equation*}
J(\NBC(M))=\langle p_C~|~ C\mbox{ is a circuit of }M\rangle.
\end{equation*}
The above discussion directly leads to the following result.
\begin{proposition}\label{prop:basis}
If $M$ is a regular matroid then
\begin{equation*}
I_{\NBC(M)}+(\Theta^{B})= J(\NBC(M)).
\end{equation*}
In particular,
\begin{equation*}
k(\NBC(M))\cong k[x_{1},\ldots,x_{n-r}]/J(\NBC(M)).
\end{equation*}
\end{proposition}
In the following we write $\Mon(n)$ for the set of monomials in $k[x_1,\ldots,x_n]$.
It was shown by Macaulay \cite{Macaulay} that a basis of $k(\NBC(M))$ can be chosen as a lower order ideal of monomials.
Recall that a subset $L\subseteq \Mon(n)$ is a \emph{lower order ideal} if whenever $u\in L$ and $v\in \Mon(n)$ divides $u$, then also $v\in L$.
Similarly, a subset $U\subseteq \Mon(n)$ is an \emph{upper order ideal} if whenever $u\in U$ and $v\in \Mon(n)$ is divisible by $u$ then $v\in U$ as well.
Note, that $U$ is an upper order ideal if and only if $\Mon(n)\setminus U$ is a lower order ideal. If $S\subseteq \Mon(n)$ then the lower and upper order ideals generated by $S$ are
\begin{align*}
L(S)&=\{u\in \Mon(n)~|~u\mbox{ divides } v \mbox{ for some } v\in S\}\;\;\mbox{  and}\\
U(S)&=\{u\in \Mon(n)~|~ u\mbox{ is divisible by } v \mbox{ for some } v\in S\}.
\end{align*}
We now construct a special lower order ideal in $k(\NBC(M))$. As is Section \ref{subsect:CircuitCocircuit} we label the fundamental circuits of $M$ w.r.t. $B$ with $C_1,\ldots,C_{n-r}$, i.e., $C_i=\ci(B,e_{i})$ for $1\leq i\leq n-r$.
The remaining circuits of $M$ are labelled with $C_{n-r+1},\ldots, C_m$.
Following \cite{BrownSagan} we define
\begin{align*}
d_j=
\begin{cases}
j&\mbox{ if } j\leq n-r\\
\min\{i~|~e_i\in \coc(B,e_j)\}&\mbox{ if }j\geq n-r+1.\\
\end{cases}
\end{align*}
Further, let
\begin{align*}
m_{C_i}=
\begin{cases}
x_i^{|\overline{C}_i|} &\mbox{ if } i\leq n-r\\
\prod_{e_j\in \overline{C}_i}x_{d_j} &\mbox{ if } i\geq n-r+1.\\
\end{cases}
\end{align*}
The following lemma states how the polynomials $p_C$ and $m_C$ are related to each other.
\begin{lemma}\label{lem:term}
Let $C$ be a circuit of $M$. Then the monomial $m_C$ occurs (up to sign) as a term in the polynomial $p_C$.
\end{lemma}
\begin{proof}
Let $C_i$ be a circuit of $M$. We first consider the case that $C_i$ is a fundamental circuit of $M$, i.e., $i\leq n-r$. In this case $m_{C_i}=x_i^{|\overline{C}_i|}$. In addition, substituting with Equation (\ref{eq:red}) in $x_{\overline{C}_i}:=\prod_{e\in \overline{C}_i}x_e$ gives $p_{C_i}=\prod_{e\in \overline{C}_i}(-\sum_{\substack{e_j\in \coc(B,e)\\e_j\neq e}}a^*_{e,j} x_j)$. Since by Lemma \ref{lem:CoCircuits} it holds that $e_i\in\coc(B,e)$ for $e\in \overline{C}_i$ the monomial $m_{C_i}$ appears as a term in $p_{C_i}$.\\
Let now $i\geq n-r+1$. Then, by definition $m_{C_i}=\prod_{e_j\in \overline{C}_i}x_{d_j}$. Accessorily, 
\begin{align*}
p_{C_i}&=\prod_{\substack{e_j\in \overline{C}_i\\ j\leq n-r}}x_j\cdot \prod_{\substack{e_j\in \overline{C}_i\\ j\geq n-r+1}}(-\sum_{e_l\in \coc(B,e_j)}a^*_{jl}x_l)\\
&=\prod_{\substack{e_j\in \overline{C}_i\\ j\leq n-r}}x_{d_j}\cdot \prod_{\substack{e_j\in \overline{C}_i\\ j\geq n-r+1}}(-x_{d_j}-\sum_{e_l\in \coc(B,e_j)\setminus\{e_{d_j}\}}a^*_{jl}x_l).
\end{align*}
For the last equality we have used that $d_j=j$ for $j\leq n-r$ and that in the sum $\sum_{e_l\in \coc(B,e_j)}a^*_{jl}x_l$ there must in particular occur the minimal element of the cocircuit.
The claim follows.
\end{proof}
In view of the above lemma it seems natural to define the following lower order ideal as a candidate for a monomial basis of $k(\NBC(M))$. We set
\begin{align}
U(M)&=U(m_{\overline{C}}~ |~ C\mbox{ circuit of } M)\qquad \mbox{ and} \notag\\
L(M)&=\Mon(n)\setminus U(M).\label{Eq:LowOrdId}
\end{align}
Note that all definitions made so far depend on the chosen ordering of the ground set $E$ of $M$ as well as on the chosen basis.
This fact is also accounted for by the following definition.
\begin{definition}\label{def:nbc}
A matroid $M$ on ground set $E(M)$ for which there is an ordering of $E(M)$ such that $L(M)$ is a basis for $k(\NBC(M))$ will be referred to as a matroid having a \emph{no broken circuit basis} or \emph{NBC basis}, for short.
\end{definition}
Note that in this context the term NBC basis refers to the set of monomials defined in (\ref{Eq:LowOrdId}) and not as usual -- in the context of matroids -- to a basis of a matroid which does not contain a broken circuit, i.e., to a facet of $\NBC(M)$. We hope that this ambiguity will not cause any confusion.
We conclude this section with a hint to a future direction of research and a question which in view of Lemma \ref{lem:term} seems natural to ask.
\begin{question}
Does there exist a term order $\preceq$ with respect to which the monomials $m_C$ are the standard monomials?
\end{question}

\subsection{Deletion and contraction of a matroid}
The aim of this section is to show a deletion-contraction property for the existence of  NBC bases. 
As in the previous section the employed methods are similar to the ones used in \cite{BrownSagan}.\\
We first recall the definition of the deletion and the contraction of a matroid in an element. Let $M=(E(M),\I(M))$ be a matroid and let $e\in E(M)$. The matroid $M\setminus e=(E(M)\setminus\{e\},\I(M\setminus e))$, where $\I(M\setminus e)=\{I\in I(M)~|~e\notin I\}$ is called the \emph{deletion of $M$ in $e$}. Let further $M/e$ be the matroid on ground set $E(M)\setminus \{e\}$ whose bases are sets $B\subseteq (E(M)\setminus\{e\})$ such that $B\cup\{e\}$ is a basis of $M$. The matroid $M/e$ is called the \emph{contraction of $M$ in $e$}.
If $e$ is neither a loop nor a coloop the $h$-polynomials of the NBC complexes of $M$, $M/e$ and $M\setminus e$ are related in the following way (see e.g., \cite{SwartzMatroid})
\begin{equation*}
h^{\NBC(M)}(t)=h^{\NBC(M\setminus e)}(t)+h^{\NBC(M/e)}(t).
\end{equation*}
Since the NBC complexes of $M$ and $M\setminus e$ are of the same dimension and the one of $M/e$ is of one dimension less this implies that
\begin{equation}\label{eq:h}
h_i^{\NBC(M)}=h_i^{\NBC(M\setminus e)}+h_{i-1}^{\NBC(M/e)}
\end{equation}
for $i\geq 0$. (Here, we set $h_{-1}^{\NBC(M/e)}=0$.)
Even though from this it seems promising to show the existence of an NBC basis by induction over the cardinality of $E(M)$ -- as Brown and Sagan did for graphs -- we encounter some difficulties.
If $B$ is a basis of $M$ and $e\notin B$ then $B\setminus \{e\}$ is a basis of $M\setminus e$ but not of $M/e$. For the contrary if $e\in B$ then $B\setminus \{e\}$ is a basis of $M/e$ but not of $M\setminus e$.
In \cite{BrownSagan} those issues were circumvented by considering only graphs which contain a vertex of degree $2$.
In the situation of matroids -- using the characterization of cocircuits provided by Proposition \ref{prop:cocircuits} -- the right assumption to meet is to require that the matroid has a cocircuit of cardinality $2$. \\
Before we can state and prove our precise result we need to give a preparatory lemma.

\begin{lemma} \label{Lem:BasisExchange}
Let $M=(E(M),\I(M))$ be a matroid and let $\{e,f\}$ be a cocircuit of $M$. Then:
\begin{itemize}
\item[(i)] A set $B\subseteq E(M)$ with $e\in B$ 
is a basis of $M$ if and only if $(B\setminus\{e\})\cup\{f\}$ is a basis of $M$.
\item[(ii)] A set $C^*\subseteq E(M)$ with $e\in C^*$ and $f\notin C^*$ is a cocircuit of $M$ if and only if $(C^*\setminus\{e\})\cup\{f\}$ is a cocircuit of $M$.
\item[(iii)] If $C\subseteq E(M)$ is a circuit of $M$ then $e\in C$ if and only if $f\in C$.
\end{itemize}
\end{lemma}
\begin{proof}
We first prove (i). Let $B$ be a basis of $M$ with $e\in B$ and let $\widetilde{B}$ be another basis of $M$ with $f\in \widetilde{B}$. We show the claim by induction on the number of elements in which those two bases differ. First assume that -- except of $e$ and $f$ -- $B$ and $\widetilde{B}$ have all but one element in common, i.e.,
\begin{align*}
B&=\{e_1,\ldots,e_s,u,e\} \qquad \mbox{and}\\
\widetilde{B}&=\{e_1,\ldots,e_s,v,f\}
\end{align*}
for $e_1,\ldots,e_s,u,v\in E(M)$. By Lemma 1.2.2 in \cite{Oxley} there exists an element $x\in (B\setminus\widetilde{B})=\{u,e\}$ such that $(\widetilde{B}\setminus\{v\})\cup\{x\}$ is a basis of $M$. If $x=u$, we obtain 
$(B\setminus\{e\})\cup\{f\}$ as a basis and we are done in this case. If not, a second application of the same Lemma to $(\widetilde{B}\setminus\{v\})\cup\{e\}=\{e_1,\ldots,e_s,e,f\}$ implies that we can exchange $e$ by $u$, thus yielding the desired basis.
If $B$ and $\widetilde{B}$ differ in a larger number of elements, we can exchange elements step by step and conclude by induction.

We now show (ii). Let $C^*\subseteq E(M)$ be a cocircuit of $M$ containing $e$. By Proposition \ref{prop:cocircuits} so as to show that $\widetilde{C}=(C^*\setminus\{e\})\cup\{f\}$ is a cocircuit of $M$ we need to show that its intersection with every basis is non-empty. Let $B$ be a basis of $M$. If $f\in B$ we have $f\in (B\cap \widetilde{C})$. If not, we have $e\in B$ since $\{e,f\}$ is a cocircuit of $M$. (i) now implies that
$\widetilde{B}=(B\setminus\{e\})\cup\{f\}$ is a basis and since $C^*$ is a cocircuit there exists $v\in (\widetilde{B}\cap C^*)$. Since $e\notin \widetilde{B}$ we further have $v\neq e$ and by $f\notin C^*$ it also holds that $v\neq f$. Thus $v\in (B\setminus\{e\})$ and $v\in (C^*\setminus\{e\})\subsetneq \widetilde{C}$. Hence, $v\in (\widetilde{C}\cap B)$. The set $\widetilde{C}$ has to be minimal since otherwise we may apply the arguments just used to a subset of $\widetilde{C}$ and thereby construct a cocircuit of $M$ which is a proper subset of $C^*$. Hence, we arrive at a contradiction. This finishes the proof of (ii).

We now prove (iii). Let $C\subseteq E(M)$ be a circuit of $M$ and let $e\in C$. Assume that $f\notin C$. By the minimality of a circuit we know that $C\setminus\{e\}$ is an independent set in $M$ which can be extended to a basis $B$ of $M$. Since $\{e,f\}$ is a cocircuit it intersects $B$ non-trivially. Since 
$e\in C$ this means that we have $f\in B$. (i) now implies that $(B\setminus\{f\})\cup\{e\}$ is a basis of $M$. Since $C\subseteq (B\setminus\{f\})\cup\{e\}$
this is a contradiction and the claim follows.
\end{proof}
Given a rank $r$ matroid $M=(E(M),\I(M))$ with a cocircuit $\{e,f\}$ we fix an ordering $e_1<\cdots< e_n$ on $E$ such that the last $r$ elements form a basis $B$. Such an ordering will be called a \emph{standard ordering} in the following. Since $\{e,f\}$ is a cocircuit and as such minimal we may assume that $B$ does only contain one of $e$ and $f$, say $e$. Without loss of generality let the ordering of $E(M)$ be such that $e_n=e$ and $e_{n-r}=f$.
By definition of the contraction and deletion and by Lemma \ref{Lem:BasisExchange} (i) the sets $(B\setminus\{e_n\})$ and $(B\setminus\{e_n\})\cup\{e_{n-r}\}$ are bases of $M/e$ and $M\setminus e$, respectively.
Thus the induced orderings $e_1,\ldots , e_{n-1}$ on $M/e_n$ and $M\setminus e_n$ are standard, as well.
\begin{theorem} \label{Th:DelContr}
Let $M=(E(M),\I(M))$ be a regular matroid of rank $r$ with a standard ordering $e_1<\cdots <e_n$. Let further $\{e_n,e_{n-r}\}$ be a cocircuit of $M$.  
If $k(\NBC(M\setminus e_n))$ and $k(\NBC(M/e_n))$ have NBC bases in their induced standard ordering, then so does $k(\NBC(M))$.
\end{theorem}
The proof of the above theorem closely follows the lines of the corresponding result in \cite{BrownSagan} for graphs. We will therefore only state the differences and emphasize where further arguments are needed. Even though throughout the proof we will assume that we are working over a field of characteristic $2$ this is not necessary. If the field is of arbitrary characteristic one has to take care of the signs appearing in the elements of the l.s.o.p. and one needs to check that signings for $M\setminus e_n$ and $M/e_n$ which are induced by compatible signings of $M$ are compatible. This essentially follows from Lemma \ref{Lem:BasisExchange}.
\begin{proof}
Assume $\chara(k)=2$.
We first show that
\begin{equation} \label{Eq:LowerOrderIdeal}
L(M)=L(M\setminus e_n )\dotcup x_{n-r} L(M/e_n).
\end{equation}
As soon as this is shown we conclude that
\begin{align*}
|L(M)|&=|L(\NBC(M\setminus e_n ))|+|L(\NBC(M/e_n))|\\
&=\sum_{i\geq 0}h_i^{\NBC(M\setminus e_n )}+\sum_{i\geq 0}h_i^{\NBC(M/e_n)}\\
&\stackrel{(\ref{eq:h})}{=}\sum_{i\geq 0}h_i^{\NBC(M)}=\dim_k k(\NBC(M))
\end{align*}
Let $B=\{e_{n-r+1},\ldots,e_n\}$, $B_1=\{e_{n-r},\ldots,e_{n-1}\}$ and $B_2=\{e_{n-r+1},\ldots,e_{n-1}\}$ and let $\Theta^B$, $\Theta^{B_1}$ and $\Theta^{B_2}$ be the l.s.o.p. for $\NBC(M)$, $\NBC(M\setminus e_n)$ and $\NBC(M/e_n)$ defined by (\ref{eq:lsop}), respectively. By Proposition \ref{prop:basis} a basis for $k(\NBC(M\setminus e_n))$ can be chosen inside $\Mon(n-r)$.
First note that the circuits of $M\setminus e_n $ are the circuits of $M$ not containing $e_n$.
Lemma \ref{Lem:BasisExchange} (iii) implies that those are the circuits of $M$ not containing $e_{n-r}$.
Thus, setting $x_{n-r}=0$ in the generators of the Stanley-Reisner ideal of $\NBC(M)$ yields the Stanley-Reisner ideal of $\NBC(M\setminus e_n )$.

Now, we consider the fundamental cocircuits of $M$ and $M\setminus e_n $. 
We claim that for $e_j\in B_1\setminus\{e_{n-r}\}$:
\begin{equation}\label{eq:cocircuits}
\coc(B_1, e_j)=\coc(B,e_j)\setminus\{e_{n-r}\}.
\end{equation}
Consider $\coc(B_1, e_j)$ in $M\setminus\{e_n\}$. Since $e_{n-r}\in B_1$ it holds that $e_{n-r}\notin \coc(B_1, e_j)$. By \cite[3.1.17]{Oxley} either $\coc(B_1, e_j)\cup\{e_n\}$ is a cocircuit of $M$ or $\coc(B_1, e_j)$ itself is a cocircuit of $M$. In the latter case it directly follows from the uniqueness of a fundamental cocircuit that $\coc(B,e_j)=\coc(B_1, e_j)$ and since $e_{n-r}\notin \coc(B,e_j)$ this shows (\ref{eq:cocircuits}) in this case. If $\coc(B_1, e_j)\cup\{e_n\}$ is a cocircuit of $M$ Lemma \ref{Lem:BasisExchange} (ii) implies that $\coc(B_1, e_j)\cup\{e_{n-r}\}$ is a cocircuit of $M$. Since $e_j\in \coc(B_1, e_j)$ and $(\coc(B_1, e_j)\cup\{e_{n-r}\})\subseteq \{e_1,\ldots,e_{n-r}\}$ this cocircuit has to be the fundamental one of $M$ induced by $e_j$ w.r.t. $B$. This finally shows (\ref{eq:cocircuits}).
Thus we get the linear form $\theta_{e_j}$ in $\Theta^{B_1}$ corresponding to $\coc(B_1, e_j)$ by setting $x_{n-r}=0$ in the corresponding linear form in $\Theta^B$ given by $\coc(B, e_j)$.
It is straightforward to verify that $\coc(B_1,e_{n-r})=\{e_{n-r}\}=\coc(B,e_n)\setminus\{e_n\}$ which takes care of the last cocircuit of $M\setminus e$ to be considered.
Together with the aforementioned relation between the Stanley-Reisner ideals of $\NBC(M)$ and $\NBC(M\setminus e_n )$ we get that the generators of $J(\NBC(M\setminus e_n))$ are obtained from those of $J(\NBC(M))$ by setting $x_{n-r}=0$. This implies that the monomials in $U(M\setminus e_n )$ and $L(M\setminus e_n)$ are the monomials in $U(M)$ and $L(M)$, respectively which are not divisible by $x_{n-r}$.

Equation (\ref{Eq:LowerOrderIdeal}) follows if we show that the monomials of $L(M)$ which are divisible by $x_{n-r}$ are given by the monomials in $L(M/e_n)$ multiplied by $x_{n-r}$.
Proposition 3.1.11 in \cite{Oxley}
states that the circuits of $M/e_n$ are the minimal non-empty members of the set
\begin{equation} \label{Eq:Circuits}
\C=\{C\setminus\{e_n\}~|~C \in\C(M)\}.
\end{equation}
If $C$ is a circuit of $M$ then $C\setminus\{e_n\}\in \C$. If the latter one is minimal in $\C$ it is a circuit of $M/e_n$. If not there exists a circuit $\widetilde{C}\subseteq C\setminus\{e_n\}$ of $M/e_n$. Hence, either $\widetilde{C}$ itself or $\widetilde{C}\cup\{e_n\}$ is a circuit of $M$. The former contradicts the minimality of $C$. So,  $\widetilde{C}\cup \{e_n\}$ has to be a circuit of $M$. Again by the minimality of $C$ it holds that $e_n\notin C$. From Lemma \ref{Lem:BasisExchange} (iii) it follows that $e_{n-r}\in (\widetilde{C}\cup\{e_n\})$. This implies $e_{n-r}\in \widetilde{C}\subsetneq C$ and by Lemma \ref{Lem:BasisExchange} (iii) we conclude
$e_n\in C$ which contradicts our assumption.
Thus we have shown that $\C$ is the set of circuits of $M/e_n$.
Following the proof of Theorem 3.2 in \cite{BrownSagan}
we call the circuits of $M$ containing $e_n$ (and $e_{n-r}$) and the corresponding ones in $M/e_n$ \emph{type I} circuits whereas circuits of $M$ not containing $e_n$ (and thus not $e_{n-r}$) and the corresponding ones in $M/e_n$ will be referred to as \emph{type II} circuits.

We now examine the fundamental cocircuits of $M/e_n$ and $M$. We claim that for $e_j\in B_2$ it holds that
\begin{equation} \label{eq:coc}
\coc(B_2,e_j)=\coc(B,e_j).
\end{equation}
Since $e_n\notin \coc(B,e_j)$ it follows from 3.1.17 in \cite{Oxley} that $\coc(B,e_j)$ is a cocircuit of $M/e_n$. From $e_j\in \coc(B,e_j)\subseteq \{e_1,\ldots,e_{n-r}\}\cup\{e_j\}= E(M)\setminus(B_2\cup\{e_n\})\cup\{e_j\}$ we deduce that $\coc(B,e_j)$ is the fundamental cocircuit of $M/e_n$ induced by $e_j$ w.r.t. $B_2$ which shows (\ref{eq:coc}).\\
We use $p_{C}$ and $\widetilde{p}_C$ to denote the generators of $J(\NBC(M))$ and $J(\NBC(M/e_n))$, respectively, and $m_C$ and $\widetilde{m}_C$ to denote the corresponding generators of $U(M)$ and $U(M/e_n)$, respectively. 
The same arguments as in the proof of Theorem 3.2 in \cite{BrownSagan} show that
\begin{align*}
p_{C}=
\begin{cases}
x_{n-r}\widetilde{p}_{C\setminus\{e_n\}} \qquad &\mbox{if $C$ is of type I}\\
\widetilde{p}_{C} \qquad &\mbox{if $C$ is of type II}.
\end{cases}
\end{align*}
Note that we have used for the type I generators that $e_{n-r}+e_n\in \Theta^B$.
In addition it holds that $m_C=x_{n-r}\widetilde{m}_{C\setminus\{e_n\}}$ if $C$ is of type I and $m_C=\widetilde{m}_C$ if $C$ is of type II.\\
Literally the same arguments as in the proof of Theorem 3.2 in \cite{BrownSagan}
show that $x_{n-r}L(M/e_n)$ consists of those monomials in $L(G)$ which are divisible by $x_{n-r}$.\\
This concludes the proof of Equation (\ref{Eq:LowerOrderIdeal}). The statement of the theorem follows if one shows that the monomials is $L(M)$ span the quotient ring $k(\NBC(M))$.
This is done in exactly the same way as in the proof of Theorem 3.2 in \cite{BrownSagan} and is therefore omitted.
\end{proof}

\section{Classes of matroids having an NBC basis} \label{sec:ThetaPhi}
This section is devoted to the study of classes of matroids which admit NBC bases.
Let us first recall the notion of a uniform matroid $U(r,n)$. Let $E=\{e_1,\ldots,e_n\}$ be an $n$-element set and let ${E\choose r}$ denote the $r$-element subsets of $E$.
Then ${E\choose r}$ is the set of bases of a matroid $U(r,n)$ and this one is called the \emph{uniform matroid} of rank $r$ on an $n$-element set. 
\begin{example}\label{Ex:uniform}
Let $M=U(n,n)$ for some $n\in \mathbb{N}$. By definition there does not exist any circuit in $M$, i.e., $M$ has no broken circuit and therefore the NBC complex is the $(n-1)$-dimensional simplex. The Stanley-Reisner ring of $\NBC(M)$ is thus the entire polynomial ring in $n$ variables. And since each element of $[n]$ is a cocircuit of $M$ the distinguished l.s.o.p. is the set of variables. This implies that quotient of the Stanley-Reisner ring is just the field and $U(n,n)$ trivially has an NBC basis, the lower order ideal being just $\{1\}$.
\end{example}

\begin{example}\label{Ex:Circuit}
Let $M=U(n-1,n)$ for some $n\in \mathbb{N}$ and let $E(M)=[n]$. Then, $[n]$ is the only circuit of $M$, hence $\{2,\ldots,n\}$ is the only broken circuit and $I_{\NBC(M)}=\langle x_2\cdot\cdots\cdot x_n\rangle$. The NBC complex is thus the cone over the boundary of an $(n-2)$-simplex with apex $1$. Furthermore, for each $1\leq i<j\leq n$ the set $\{i,j\}$ is a cocircuit of $M$. For simplicity we now assume that we are working over a field of characteristic $2$, otherwise we would need to take care of the signs in the following discussion. If we use $\{2,\ldots,n\}$ as the basis in the construction of the distinguished l.s.o.p. then $\Theta=\{x_1+x_j~|~2\leq j\leq n\}$ is the l.s.o.p. and by substituting we have $x_2\cdot\cdots\cdot x_n=x_1^{n-1}$ in $I_{\NBC(M)}\cup\Theta$. Thus, the lower order ideal we are looking for is $\{1,x_1,x_1^2,\ldots,x_1^{n-2}\}$ and since the $h$-vector of the broken circuit complex is $(1,\ldots,1,0)$ it follows that $U(n-1,n)$ has an NBC basis.
\end{example}
The operation we are primarily concerned with in this section is what is usually referred to as the parallel connection of two matroids with respect to a basepoint. Let $M=(E(M),\I(M))$, $N=(E(N),\I(N))$ be two matroids and let $p\in (E(M)\cap E(N))$. Consider the set 
\begin{equation*}
\C_P=\C(M)\;\cup \;\C(N)\;\cup \;\{(C_1\cup C_2)\setminus\{p\}~|~p\in (C_1\cap C_2),\; C_1\in\C(M),\; C_2\in\C(N)\}.
\end{equation*}
It is straightforward to verify that $\C_P$ is the set of circuits of a matroid $P(M,N;p)$, the so-called \emph{parallel connection of $M$ and $N$ w.r.t. the basepoint $p$}.
Even though different choices of the basepoint in general lead to different matroids we suppress the basepoint from our notation if it is clear which one is used.
Iterating this construction we can define the parallel connection of several (more than $2$) matroids. More precisely, if $M_1,\ldots, M_t$ are matroids and if $p\in E(M_1)\cap\cdots \cap E(M_t)$ then we set $P_1=M_1$ and $P_i=P(P_{i-1},M_i;p)$ for $2\leq i \leq t$ and $P(M_1,\ldots,M_t;p)$ is referred to as the 
\emph{parallel connection of $M_1,\ldots,M_t$ w.r.t. the basepoint $p$}.
The following lemma specifies how the bases of the parallel connection look like.
This is achieved by the following lemma which makes repeated use of Proposition 7.1.13 (ii) in \cite{Oxley}.
\begin{lemma}\label{lem:ParCon}
Let $M=P(M_1,\ldots, M_t)$ where $M_1,\ldots M_t$ are such that $E(M_i) \cap E(M_j) = \{p\}$ for all $i\neq j$ and assume that $p$ is a loop in none of the $M_i$. Then $B \subseteq E(M_1) \cup \cdots \cup E(M_t)$ is a basis of $M$ containing $p$ if and only if $B \cap E(M_i)$ is a basis of $M_i$ containing $p$ for all $i$. Furthermore, $B$ is a basis of $M$ not containing $p$ if and only if for exactly one $i$, the set $B \cap E(M_i)$ is a basis of $M_i$ not containing $p$ and $(B \cap E(M_j)) \cup \{p\}$ is a basis of $M_j$ for all $j \neq i$.
\end{lemma}
\begin{proof}
We show the claim by induction over the number of components of $M$. If $t=2$ the claim directly follows from Proposition 7.1.13 (ii) in \cite{Oxley}. If $t\geq 3$ set $P_{t-1}=P(M_1,\ldots, M_{t-1})$. Then $M=P(P_{t-1},M_t)$ and by the base of the induction $B$ is a basis of $M$ containing $p$ if and only if $B\cap (E(M_1)\cup\cdots\cup E(M_{t-1}))$ is a basis of $P_{t-1}$ containing $p$ and $B\cap E(M_t)$ is a basis of $M_t$ containing $p$.
By the induction hypothesis the former condition is equivalent to $p\in B\cap E(M_i)$ being a basis of $M_i$ for $1\leq i\leq t-1$.
The second part follows by a similar reasonning.
\end{proof}
As a direct consequence of Lemma \ref{lem:ParCon} we obtain:
\begin{corollary}
Let $M_1$, $M_2$ and $M_3$ be matroids with $E(M_i) \cap E(M_j) = \{p\}$ for $i\neq j$. Assume that $p$ is a loop in none of the $M_i$. Then the parallel connection of the matroids $M_1, M_2, M_3$ is `associative', i.e.,
\begin{equation} \label{Eq:ass}
P(M_1,P(M_2,M_3)) = P(P(M_1,M_2),M_3).
\end{equation}
\end{corollary}
We want to show that parallel connections of certain types of matroids possess NBC bases. For this purpose Theorem \ref{Th:DelContr} will be used which makes it necessary to understand the effect of deletion and contraction of an element from a parallel connection. This is afforded by the next proposition.
\begin{proposition}\cite[Proposition 7.1.15 (v)]{Oxley}\label{pr:parallel}
Let $M=P(M_1,\ldots, M_t;p)$ where $M_1,\ldots M_t$ are such that $E(M_i) \cap E(M_j) = \{p\}$ for $i\neq j$. Let $e \in E(M_1) \setminus p$. Then
\begin{equation*}
M \setminus e = P(M_1\setminus e,M_2,\ldots, M_t) \qquad \mbox{and} \qquad M / e = P(M_1/e,M_2,\ldots,M_t).
\end{equation*}
\end{proposition}
We are now able to state one of the main results of this section.
\begin{theorem}\label{Th:phi}
Let $M_i = U(n_i-1, n_i)$ and $n_i\geq 2$ for $1\leq i \leq t$. Assume that $E(M_i)\cap E(M_j)=\{p\}$ for all $i\neq j$.
Then $M=P(M_1,\ldots, M_t;p)$ has an NBC basis.
\end{theorem}

\begin{proof}
Throughout the proof we assume that $n_1 \leq n_2 \leq \cdots \leq n_t$.
We need to label the $n$ elements of $E(M)$ so that $e_1 < e_2 < \cdots < e_n$ is a standard ordering. Note that $n=\sum_{i=1}^t n_i - t + 1$. We can obtain a basis for $M$ by removing $p$ and an element (not $p$) from all but one of the $E(M_i)$.
Let $e_1=p$. For $2 \leq i \leq t+1$ let $e_i$ be an element of $E(M_{i-1})$ such that $\{e_i, p\}$ is a cocircuit of $M_{i-1}$. (Note that each $2$-element subset of $E(M_i)$ is a cocircuit.) Then we label the remaining elements $e_{n}, e_{n-1},\ldots, e_{t+2}$ in any order such that the elements of $E(M_1)\setminus \{p,e_2\}$ come first, the elements of $E(M_2) \setminus \{p,e_3\}$ next, and so on.

Let us call a matroid of this type a \emph{theta matroid} and a labelling of this type a \emph{theta labelling}.

Note that $e_n$ is in the first matroid $M_i$ of size at least $3$, and so $\{e_n, e_{i+1}\}$ is a cocircuit of $M_i$ and consequently of $M$. We use double induction on the number of elements $n$ of $M$ and the number of components of the parallel connection.
If $t=1$ then $M=U(n_1-1,n_1)$ and it follows from Example \ref{Ex:Circuit} that $M$ has an NBC basis in this case.\\
Now assume $t\geq 2$.
We consider the matroid $EM / e_n$. From Proposition \ref{pr:parallel} combined with $U(n-1,n)/e = U((n-1)-1,n-1)$, see \cite[Example 3.1.0]{Oxley}, we deduce that $M/e_n$ is a theta matroid. Its labelling is the induced one and thus a theta labelling. Thus, $M / e_n$ has an NBC basis by the induction hypothesis.

In general, $M \setminus e_n$ is not a theta matroid, but it is a theta matroid with the induced labelling a theta labelling if we ignore the other elements of $M_i$. (Each of these is a cocircuit of $M\setminus e_n$ and contributes nothing to $k(\NBC(M\setminus e_n))$.) This shows that $M \setminus e_n$ has an NBC basis by the induction hypothesis.
The claim now follows from Theorem \ref{Th:DelContr}.
\end{proof}
The matroids treated in the above theorem were named theta matroids since they are the matroid theoretic analogue of theta graphs considered in \cite{BrownSagan}.

The only uniform matroids which are regular are those of rank $0$ or rank $1$ or those of the form $U(n-1,n)$ and $U(n,n)$, see e.g., the appendix of \cite{Oxley}. One could therefore also think about looking at parallel connections of the latter class. However, as the next example shows, this class turns out to be uninteresting.

\begin{example}
Example \ref{Ex:uniform} shows that the matroid $U(n,n)$ has an NBC basis.
The parallel connection of $U(n_1,n_1)$ and $U(n_2,n_2)$ has rank $n_1+n_2-1$ and since $|E(P(U(n_1,n_1),U(n_2,n_2)))|=n_1+n_2-1$ it is again a uniform matroid. It therefore has an NBC basis and thus this property is preserved under taking parallel connections in this case. However, it does not enrich the class of matroids with this property since it acts trivially on this type of uniform matroids.
\end{example}

Having investigated parallel connections of matroids w.r.t. one basepoint one might want to look at parallel connections w.r.t. several basepoints. If $M_1,\ldots, M_t$ are matroids and $p_i\in E(M_i)\cap E(M_{i+1})$ for $1\leq i\leq t-1$ we set $P_1=M_1$ and $P_i=P(P_{i-1},M_{i};p_{i-1})$ for $2\leq i\leq t$. We write $P(M_1,\ldots,M_t;p_1,\ldots,p_{t-1})$ for $P_t$ and call this the \emph{parallel connection of $M_1,\ldots, M_t$ w.r.t. $p_1,\ldots, p_t$}.
\begin{theorem}\label{Th:theta}
Let $M_i = U(n_i-1, n_i)$ and $n_i\geq 2$ for $i \in \{1,\ldots, t\}$. Assume $E(M_{i})\cap E(M_{i+1})=\{p_{i+1}\}$ for $1\leq i\leq t-1$ and $p_i\neq p_j$ for $i\neq j$. Then $M=P(M_1,\ldots, M_t;p_2,\ldots,p_t)$ has an NBC basis.
\end{theorem}
\begin{proof}
As in Theorem \ref{Th:phi} we first need to introduce a labelling of $E(M)$ that is a standard ordering. First choose an element $e_t\in E(M_1)$ with $e_t\neq p_2$. Further set $e_i=p_{t+1-i}$ for $1\leq i\leq t-1$. Label the remaining elements of $E(M)$ by $e_{t+1},e_{t+2},\ldots, e_{n}$ where the missing elements of $E(M_t)$ come first, those of $E(M_{t-1})$ next and those of $E(M_1)$ last. 
We call a parallel connection of uniform matroids $U(n_i-1,n_i)$ with the introduced labelling a \emph{phi matroid} with \emph{phi labelling}.
We proceed by induction over the cardinality of the ground set $E(M)$ and the number of components of $M$. The cases $t=1$ and $t=2$ directly follow from Theorem \ref{Th:phi}. For the induction step note that if we assume $n_1\leq n_2\leq \cdots\leq n_t$ we have $e_n\in M_i$ for the first matroid of size $3$ again. The claim then follows using arguments similar to those in the proof of Theorem \ref{Th:phi} and by the application of Theorem \ref{Th:DelContr}.
\end{proof}
Although the last two results give examples of classes of regular matroids which possess NBC bases it would be interesting to know if there exist other classes of matroids which do have this property.
Brown and Sagan \cite{BrownSagan} conjectured that all graphic matroids possess an NBC basis. However, this would not be a complete characterization of matroids having this property since there are examples of matroids which are not graphic but have an NBC basis.
\begin{example}
The dual matroid $M^*(K_{3,3})$ of the graphic matroid corresponding to the complete bipartite graph $K_{3,3}$ is regular as well as cographic (i.e., its dual is graphic) but not graphic.
A representation (over any field) for $M(K_{3,3})$ is given by the vertex-edge incidence matrix of $K_{3,3}$:
\[
\begin{pmatrix}
1 & 1 & 1 & 0 & 0 & 0 & 0 & 0 & 0\\
0 & 0 & 0 & 1 & 1 & 1 & 0 & 0 & 0\\
0 & 0 & 0 & 0 & 0 & 0 & 1 & 1 & 1\\
0 & 0 & 1 & 0 & 0 & 1 & 0 & 0 & 1\\
0 & 1 & 0 & 0 & 1 & 0 & 0 & 1 & 0\\
1 & 0 & 0 & 1 & 0 & 0 & 1 & 0 & 0
\end{pmatrix}.
\]
Following \cite{Oxley} we derive a representation for the dual matroid $M^*(K_{3,3})$
\[
\begin{pmatrix}
1 & 1 & 0 & 1 & 1 & 0 & 0 & 0 & 0\\
1 & 0 & 1 & 1 & 0 & 1 & 0 & 0 & 0\\
1 & 1 & 0 & 0 & 0 & 0 & 1 & 1 & 0\\
1 & 0 & 1 & 0 & 0 & 0 & 1 & 0 & 1
\end{pmatrix}.
\]
Reordering the columns of the above representation yields a standard ordering for $M^*(K_{3,3})$ (i.e., the last five columns
are a basis)
\[
\begin{pmatrix}
0 & 0 & 1 & 0 & 1 & 0 & 0 & 1 & 1\\
0 & 0 & 0 & 0 & 1 & 1 & 1 & 0 & 1\\
1 & 1 & 1 & 0 & 0 & 0 & 0 & 0 & 1\\
0 & 1 & 0 & 1 & 0 & 0 & 1 & 0 & 1
\end{pmatrix}.
\]
Using this ordering we get the following set of NBC monomials 
\begin{align*}
L(M^*(K_{3,3}))=\{&1, x_1, x_2, x_3, x_4, x_5, x_1^2, x_1 x_2, x_1 x_4, x_2^2 , x_2 x_3, x_2 x_5, x_3 x_4, x_3 x_5, x_4 x_5,\\
 &x_1^2 x_2, x_1^2 x_4, x_2^2 x_3, x_2^2 x_5, x_3 x_4 x_5\}
\end{align*}
Using a computer algebra system e.g., Macaulay2, \cite{M2} it can be verified that this set is indeed a basis for the quotient $k(\NBC(M^*(K_{3,3})))$.
\end{example}
Even though we have just stated an example of a regular but not graphic matroid which possesses an NBC basis -- as the following example shows -- this cannot 
be expected to be true for all regular matroids.
\begin{example}
An example of a regular matroid which is neither graphic nor cographic is $R_{10}$ which can be represented by the matrix
\[
\begin{pmatrix}
1 & 0 & 0 & 0 & 0 & 1 & 1 & 0 & 0 & 1 \\ 
0 & 1 & 0 & 0 & 0 & 1 & 1 & 1 & 0 & 0 \\ 
0 & 0 & 1 & 0 & 0 & 0 & 1 & 1 & 1 & 0 \\
0 & 0 & 0 & 1 & 0 & 0 & 0 & 1 & 1 & 1 \\
0 & 0 & 0 & 0 & 1 & 1 & 0 & 0 & 1 & 1 \\
\end{pmatrix}
\]
over any field. This is a matroid of rank five on the edge set $E(R_{10})$ of ten elements which has $162$ bases. Thus there are $162\cdot 5! \cdot (10 -5)! = 2.332.800$ standard orderings of the edge set $E(R_{10})$. Computer calculations using Sage \cite{sage} show that none of these orderings gives an NBC basis. 
\end{example}
The above examples suggests the following extension of the conjecture in \cite{BrownSagan}.
\begin{conjecture}
Let $M$ be a matroid on ground set $E(M)$. If $M$ is graphic or cographic there exists a standard ordering
of $E(M)$ such than $\NBC(M)$ has an NBC basis.
\end{conjecture}
It could be further interesting -- not only for its own right -- but also in order to identify classes of matroids which have NBC bases to find constructions and operations on matroids which preserve the property of having an NBC basis. For example one could look at arbitrary parallel and series connections.
Even though such a statement is maybe not possible in general it seems worth studying if anything can be said about minors of matroids imposing some restrictions.

\section*{acknowledgment}
We would like to thank Einar Steingr\'imsson for encouraging us to carry out this project together. Furthermore, we would like to thank Jason Brown and Dave Wagner for providing us additional information about their paper \cite{BrownColbournWagner}.
We are also grateful to Volkmar Welker for comments on an earlier version of this paper.

\bibliographystyle{alpha}
\bibliography{biblio}

\end{document}